\documentclass{amsart}
\usepackage{amssymb}
\usepackage{amscd}
\usepackage{amsmath}
\usepackage{enumerate}
\usepackage[dvips]{graphicx}
\newtheorem{theo}{Theorem}
\newtheorem{lema}[theo]{Lemma}

\newcommand{\CC}{{\mathbb{C}}}

\newcommand{\calA}{{\mathcal{A}}}
\newcommand{\calB}{{\mathcal{B}}}

\newcommand{\calX}{{\mathcal{X}}}
\newcommand{\calY}{{\mathcal{Y}}}
\newcommand{\calZ}{{\mathcal{Z}}}
\newcommand{\comp}{{\circ}}

\begin{document}
\title[]{On algebraic representatives of homeomorphism types of analytic hypersurface germs}
\author{Javier Fern\'andez de Bobadilla}
\address{ICMAT. CSIC-Complutense-Aut\'onoma-Carlos III}
\email{javier@mat.csic.es}
\thanks{Research partially supported by the ERC Starting Grant project TGASS and by Spanish Contract MTM2007-67908-C02-02. The author thanks to the Faculty de Ciencias Matem\'aticas of the Universidad Complutense de Madrid for excellent working conditions.}
\date{26-2-2010}
\begin{abstract}
A question of B. Teissier, inspired by a previous problem of R. Thom, asks whether for any germ of complex analytic hypersurface
there exists a germ of complex algebraic hypersurface with the same topological type. Up to now only the case of germs with an 
isolated singularity is kwown. In this article we answer in the affirmative the case $1$-dimensional singular set, also for
embedded topological types.
\end{abstract}


\maketitle

\section{Introduction}
In 1970 R. Thom~\cite{Th2} formulated the following problem: given any germ of (real or complex) analytic morphism, is there an 
germ of an algebraic morphism with the same weakly stratified structure? He noticed that an eventual affirmative solution would
imply that any germ of analytic set is homeomorphic to a certain germ of algebraic set. 

Inspired in Thom's question, Teissier formulated the following question, (see Question~C in~\cite{Te}) and even proposed an
approach towards its solution: given a complex analytic hypersurface germ, is it homeomorphic to a complex algebraic
hypersurface germ?.
 
In the case of hypersurfaces with an isolated singularity at the origin the answer to Thom's question is affirmative in
a much stronger sense: if $f$ is the equation defining the hypersurface, there is a germ of complex biholomorphism at
the origin which composed with $f$ gives a polynomial. Thus, the hypersurface germ is even analytically isomorphic to an
algebraic one, and the isomorphism may be extended to the ambient.

In the case of non-isolated singularities the question is, up to the author's knowledge, still open. 

Whitney gave in~\cite{Wh} 
examples of germs at the origin of analytic hypersurfaces in $\CC^3$ not analytically isomorphic to an algebraic hypersurface 
germ. One of them is defined by the following equation: 
\[xy(y-x)(y-(3+t)x)(y-e^tx)=0.\]
An easy inspection shows that the singular set of the above hypersurface is the $t$-axis. Thus, if the dimension of the singular
set is at least $1$ we can not expect to have analytic isomorphisms to algebraic germs.

Here we settle Teissier's question for hypersurface germs with $1$-dimensional singular set, by a homeomorphism that even extends
to the ambient.

\begin{theo}
\label{main}
For complex analytic hypersurface germ $(X,x)\subset(\CC^n,x)$ with $1$-dimensional singular set there is a self-homeomorphism of
$(\CC^n,x)$ taking $(X,x)$ to an algebraic hypersurface germ. 

\end{theo}

Our methods differ from Teissier's proposed approach.
The idea of the proof is a follows: taking a generic coordinate $z_1$ we view the hypersurface as a uniparametric 
deformation of the corresponding hyperplane section through the origin. The hyperplane section has isolated singularities and 
therefore may be represented by a polynomial, and moreover, has a polynomial versal deformation. The uniparametric deformation of
the hyperplane section corresponds to an analytic path $\gamma$ in the base of the versal deformation, and the
original hypersurface is the pullback of the universal family by this path.

Using Artin's approximation Theorem~\cite{Ar}, Hironaka's resolution of singularities and the ideas of Thom's First Isotopy
Lemma~\cite{Th1} we may replace the original hypersurface germ by another one, with the same embedded topological type, for 
which the associated path in the base of the versal deformation is given by algebraic power series. 
This allows to see the hypersurface germ as the local projection of an algebraic variety with perhaps higher embedding 
codimension. A generic projection argument allows to conclude.

Probably the use of Hironaka's resolution of singularities could be 
avoided at the expense of a longer and less clear exposition.

\section{The proof}
Let $(X,O)$ be a germ at the origin of analytic hypersurface in $\CC^n$. Let $f$ be its defining equation. Choose coordinates 
$(z_1,...,z_n)$ of $\CC^n$ such that the slice $V(z_1,f)$ has an isolated singularity at the origin.
The first step of the proof is to view $f$ as a family of functions in $z_2,...,z_n$ parametrised by a parameter $z_1$, and 
to modify the family fibrewise to make $f$ polynomial in the variables $z_2,...,z_n$.

\subsection{} Denote by 
\[g:=f|_{V(z_1)}:(V(z_1),O)\cong (\CC^{n-1},O)\to\CC\]
the restriction of $f$ to $V(z_1)$. Since $g$ has an isolated singularity at the origin, by finite determinacy, there exists
a germ 
\[\phi:(V(z_1),O)\to (V(z_1),O)\]
such that $g\comp\phi$ is a polynomial in $z_2,...,z_n$. Define
\[\psi:(\CC^n,O)\to (\CC^n,O)\]
by $\psi(z_1,z_2,...,z_n):=(z_1,\phi(z_2,...,z_n))$. Redefining $f$ as the composition $f\comp\phi$ we may assume that $g$ is a
polynomial in $z_2,...,z_n$ with an isolated singularity at the origin.

\subsection{} Let $\mu$ be the Milnor number of $g$ at the origin. Consider polynomials $h_1,...,h_\mu$ whose classes form a basis for the
Milnor algebra
\[\frac{\CC\{z_2,...,z_n\}}{(\partial g/\partial z_2,..,\partial g/\partial z_n)}.\]
The unfolding
\begin{equation}
\label{versalunfolding}
G:\CC^{n-1}\times\CC^\mu\to\CC
\end{equation}
defined by 
\[G(z_2,...,z_n,t_1,...,t_\mu):=g(z_2,...,z_n)+\sum_{i=1}^\mu t_ih_i(z_2,...,z_n)\]
is a versal unfolding for $g$.

\subsection{} Viewing $f$ as a $1$-parameter unfolding of $g$ parametrised by $z_1$ and using the fact that the
unfolding~(\ref{versalunfolding}) is versal we find an analytic mapping
\[\gamma:(\CC,0)\to\CC^\mu\]
and a germ of biholomorphism
\[\varphi:(\CC^n,O)\to (\CC^n,O)\]
of the form $\varphi(z_1,z_2,...,z_n)=(z_1,\varphi_{z_1}(z_2,...,z_n))$ such that we have the equality:
\begin{equation}
\label{casi}
f\comp\varphi(z_1,...,z_n)=G(z_2,...,z_n,\gamma_1(z_1),...,\gamma_\mu(z_1)),
\end{equation}
where $(\gamma_1,...,\gamma_\mu)$ is the coordinate expansion of $\gamma$.

Notice that the second term of the above equality is polynomial in all variables except possibly in $z_1$. 
We redefine $f$ to be equal to $f\comp\varphi$. In the rest of the proof we will modify $f$ so that it becomes a polynomial
defining a hypersurface with the same embedded topological type at the origin.

\subsection{}\label{whitney}
Denote by $Z\subset\CC^{n-1}\times\CC^\mu$ the hypersurface defined by $G=0$. 
Consider the projection
\begin{equation}
\label{proyeccion}
\pi:(\CC^{n-1}\times\CC^\mu,Z)\to\CC^\mu.
\end{equation}
Consider Whitney stratifications $\calA$ and $\calB$ of $\CC^{n-1}\times\CC^\mu$ and $\CC^\mu$ respectively, which make
$\pi$ a Thom stratified mapping (see~\cite{GLPW}, Chapter~1) with the following additional properties:
\begin{enumerate}[(i)]
\item the set $Z$ is a union of strata of $\calA$,
\item the origin $\{O\}$ is a stratum of $\calB$,
\item the strata of $\calA$ and $\calB$ are locally closed algebraic subsets.
\end{enumerate}

\subsection{}\label{artin} Let $C$ be the curve germ given by the image of $\gamma$. Let $d$ be the degree of the mapping 
\[\gamma:(\CC,0)\to (C,O).\]
There exists an analytic parametrisation 
\[\alpha:(\CC,O)\to (C,O)\]
such that $\gamma(z_1)=\alpha(z_1^d)$.

Let $B$ be the stratum of $\calB$ containing the generic point of $C$. Since $\overline{B}$ is an algebraic subset of $\CC^\mu$,
by Artin's Approximation Theorem~\cite{Ar}, for any positive integer $N$ there exists a parametrised curve
\[\beta:(\CC,O)\to (D,O)\subset\CC^\mu\]
such that $\beta_i(\CC)$ is contained in $\overline{B}$ and
whose coordinate functions $\beta_i$ are {\em algebraic power series} have a Taylor expansion coinciding with the Taylor 
expansion of the coordinate functions $\alpha_i$ up to degree $N$. Denote by $(D,O)$ the germ given by $(\beta(\CC),O)$.

\subsection{}\label{hironaka} Consider the algebraic subset of $\CC^\mu$ defined by $B':=\overline{B}\setminus B$. 
Since $C$ is a curve not contained in $B'$, there is a sequence of blowing ups in smooth centers 
\begin{equation}
\label{blowup}
\sigma:Y\to\CC^{\mu}
\end{equation}
such that:
\begin{enumerate}
\item The algebraic subset $\overline{B}$ is desingularised.
\item The strict transforms $\tilde{B}'$ and $\tilde{C}$ of $B'$ and $C$ are disjoint.  
\item If $\tilde{C}$ and $\tilde{B}$ are the strict transforms of $C$ and $\overline{B}$ by $\sigma$ and $E$
denotes the exceptional divisor, the germ of the triple $(\tilde{B},\tilde{B}\cap E,\tilde{C})$ at $\tilde{\alpha}(0)$ is 
biholomorphic to the germ of $(\CC^k,V(x_1),V(x_2,...,x_k))$ at the origin of $\CC^k$, where $(x_1,...,x_k)$ are the coordinates
of $\CC^k$ (here $k$ is the dimension of $B$).
\end{enumerate}
In order to achieve this we use Hironaka's Theorem on resolution of singularities~\cite{Hi}:
the composition $\sigma_1$ of a sequence of blowing ups in smooth centers located at the singular 
set produces an embedded desingularisation of $C\cup B'$. Since the total transform of $C\cup B'$ is normal crossings,
if the strict transform of $C$ meets the strict transform of $B'$ then one more blowing up at the meeting point
pulls them appart. We redefine $\sigma_1$ to include also the last blowing up if it is needed.
After doing this we desingularise the strict transform of $\overline{B}$ by $\sigma_1$ by the composition $\sigma_2$ of a 
sequence of blowing ups in smooth centers. The composition $\sigma_1\comp\sigma_2$ has properties~(1)~and~(2). 
Finally, since the generic point of $C$ belongs to the smooth locus $\overline{B}$, the strict transform of $C'$ by 
$\sigma_1\comp\sigma_2$ meets the exceptional divisor $E'$ of $\sigma_1\comp\sigma_2$ at a single point. There is an embedded 
desingularisation $\sigma_3$ of $C'\cup E'$ by a sequence of blowing ups at smooth centers such that
$\sigma:=\sigma_1\comp\sigma_2\comp\sigma_3$ has all the required properties.
  
\subsection{}\label{homotopy} Let $\tilde{\alpha}$ denote the lifting of the parametrisation $\alpha$ by $\sigma$.
If the number $N$ in~\ref{artin} is taken large enough then the generic point of the curve $D$ is contained in 
the smooth locus of $\overline{B}$ and does not meet $B'$. Then we may consider the lifting $\tilde{\beta}$ of the parametrisation
$\beta$ by the transformation $\sigma$. 
We can take $N$ large enough so that we have the equality $\tilde{\beta}(0)=\tilde{\alpha}(0)$ 
and moreover, if $(x_1,...,x_k)$ are the local coordinates of $\tilde{B}$ at $\tilde{\alpha}(0)$ given in Property~(3) then
the Taylor expansion of $x_i(\tilde{\alpha})$ and $x_i(\tilde{\beta})$ coincide up to order $1$.
In this situation both $\tilde{\alpha}$ and $\tilde{\beta}$ are parametrised curves through $\tilde{B}$ meeting transversely
the exceptional divisor $E$: in coordinates $(x_1,...,x_k)$, the Taylor development up to degree $1$ of both mappings
is given by $\alpha(t)=(\lambda t,0,...,0)$ with $\lambda\neq 0$, and the exceptional divisor is $x_1=0$.

Since $\tilde{\alpha(0)}$ does not belong to $\tilde{B}'$, we may choose a sufficiently small radius $\delta$ such that the
restriction
\[\tilde\alpha:\overline{D}_\delta\to \tilde{B}\]
\[\tilde\beta:\overline{D}_\delta\to \tilde{B}\]
of the parametrisations to the closed disc $\overline{D}_\delta$ of radius $\delta$ satisfies that
\[\tilde{\alpha}(\overline{D}_\delta)\cap E=\{\tilde{\alpha}(0)\},\quad\quad\tilde{\alpha}(\overline{D}_\delta)\cap\tilde{B}'=\emptyset,\]
\[\tilde{\beta}(\overline{D}_\delta)\cap E=\{\tilde{\alpha}(0)\},\quad\quad\tilde{\beta}(\overline{D}_\delta)\cap\tilde{B}'=\emptyset,\]
\[\tilde{\alpha}(\overline{D}_\delta)\cap\tilde{\beta}(\overline{D}_\delta)=\{\tilde{\alpha}(0)\}.\]
The last property is only satisfied if the germs $(C,O)$ and $(D,O)$ are different. Otherwise we know that, after
reparametrisation, the curve $\alpha$, and hence $\gamma$, is given by an algebraic power series. If this is the case then we jump
to Step~\ref{final}.

After this choice it is easy to find a smooth homotopy
\begin{equation}
\label{homotopiaarriba}
\tilde{H}:\overline{D}_\delta\times [0,1]\to \tilde{B}
\end{equation}
such that it is an inmersion at any point of $(\overline{D}_\delta\setminus\{O\})\times [0,1]$, its restriction to
$\overline{D}_\delta\setminus\{0\}\times [0,1]$ is injective, and we have the equalities
\[\tilde{H}(t,0)=\tilde{\alpha}(t),\quad\quad\tilde{H}(t,1)=\tilde{\beta}(t),\]
\[\tilde{H}(0,u)=\tilde{\alpha}(0)=\tilde{\beta}(0)\quad\text{for any}\quad u\in [0,1],\]
and the inclussion 
\[\tilde{H}(\overline{D}_\delta\setminus\{0\}\times [0,1])\subset \tilde{B}\setminus (E\cup \tilde{B}').\]

We push down the homotopy by $\sigma$ and define the smooth homotopy
\begin{equation}
\label{homotopiaabajo}
H:\overline{D}_\delta\times [0,1]\to B
\end{equation}
which satisfies:
\begin{enumerate}[(a)]
\item $H(t,0)=\alpha(t)$, $H(t,1)=\beta(t)$.
\item $H(0,u)=O$ for any $u\in [0,1]$.
\item The restriction of $H$ to $\overline{D}_\delta\setminus\{0\}\times [0,1]$ is an injective inmersion and we have the 
inclussion $H((\overline{D}_\delta\setminus\{0\})\times [0,1])\subset B$.
\end{enumerate}

\subsection{}\label{toptriv} Let $\epsilon$ be a positive radius such that $V(g)$ meets transversely the sphere of radius
$r$ centered at the origin of $\CC^{n-1}$ for any $0<r\leq\epsilon$. Let $\overline{B}_\epsilon$ denote the closed ball of radius 
$\epsilon$ in $\CC^{n-1}$ centered at the origin. For an small neighbourhood $S$ of the origin in $\CC^\mu$
the fibre $G^{-1}(0,s)$ meets transversely the sphere $\partial\overline{B}_\epsilon$. The restriction
\begin{equation}
\label{proyeccionrestringida}
\pi|_{\overline{B}_\epsilon\times S}:(\overline{B}_\epsilon\times S,Z)\to\CC^\mu.
\end{equation}
of the projection~(\ref{proyeccion}) becomes proper. The restrictions $\calA'$ and $\calB'$ 
of the Whitney stratifications $\calA$ and $\calB$ to 
$\overline{B}_\epsilon\times S$ and $S$ are still Whitney stratifications which make $\pi|_{\overline{B}_\epsilon\times S}$ a
Thom stratified mapping and have have Property~(i) and ~(ii) of~\ref{whitney}.

\subsection{} Let 
\begin{equation}
\label{pullback}
\rho:(\overline{B}_\epsilon\times\overline{D}_{\delta^{1/d}}\times [0,1],Z')\to \overline{D}_{\delta^{1/d}}\times [0,1]
\end{equation}
be the pullback of mapping~(\ref{proyeccionrestringida}) by $H$.
We fibre the pair $(\overline{B}_\epsilon\times\overline{D}_{\delta^{1/d}}\times [0,1],Z')$ over $[0,1]$ by the projection to the
third factor. Denote by $Z'_u$ the fibre of $Z'$ over $u\in [0,1]$ by this projection.

\begin{lema}
\label{isotopy}
There exists a family of
homeomorphisms 
\[h_u:\overline{B}_\epsilon\times H(\overline{D}_{\delta}\times\{0\})\to
\overline{B}_\epsilon\times H(\overline{D}_{\delta}\times\{u\})\]
for any $u\in [0,1]$ continuously depending on the parameter $u$ and preserving the stratification $\calA'$.
\end{lema}
\begin{proof}
We follow the ideas of the proof of Thom's First Isotopy Lemma given in~\cite{GLPW}.

Consider the vector field $\calX:=(0,0,\partial/\partial u)$ in $(\overline{D}_{\delta^{1/d}}\setminus\{0\})\times [0,1]$ 
(where $u$ is the coordinate of $[0,1]$). Let $TM$ denote the tangent bundle of any manifold $M$. Let
\[DH:T(\overline{D}_\delta\times [0,1])\to T\CC^\mu\]
be the derivative of $H$. Since $H$ is an injective inmersion at $\overline{D}_\delta\setminus\{0\}\times [0,1]$ and we have the 
vanishing $DH(x)(\calX(x))=0$ for any $x\in\{0\}\times [0,1]$ there exists a unique continuous section 
\[\calY_0\in\Gamma(H(\overline{D}_\delta\times [0,1]),T\times\CC^\mu)\]
such that $DH(x)(\calX(x))=\calY_0(H(x))$ for any $x\in\overline{D}_\delta\times [0,1]$. Notice that $\calY_0(O)=0$ and that
$\calY_0$ is tangent to $B$ at any $H(x)$ with $x\in\overline{D}_\delta\setminus\{0\}\times [0,1]$. 

Since the origin is a stratum of $\calB'$ it is easy to construct an extension $\calY$ of $\calY_0$ to a continuous 
vector field defined over a neighbourhood $U$ of the compact set $H(\overline{D}_\delta\times [0,1])$ in $\CC^\mu$ which is 
stratified and weakly controlled with respect to $\calB'$ (see Definition~I.3.1. of~\cite{GLPW}).

As $\pi|_{\overline{B}_\epsilon\times S}$ is a Thom stratified mapping and $\calY$ is weakly controlled, Theorem~II.3.2 
of~\cite{GLPW} allows to construct a vector field $\calZ$ over $(\pi|_{\overline{B}_\epsilon\times S})^{-1}(U)$ which lifts
$\calY$ and is stratified for $\calA'$ and controlled over $\calY$.

The results of Section~I.4~of~\cite{GLPW} allow to integrate $\calZ$ to a continuous flow
\[h:\overline{B_\epsilon}\times H(\overline{D}_\delta\times\{0\})\times [0,1]\to
\overline{B_\epsilon}\times H(\overline{D}_\delta\times [0,1])\]
which preserves the stratification $\calA'$. The required family of homeomorphisms is given by $h_u:=h(\centerdot,u)$. 
\end{proof}

\subsection{} Define the path
\[\tau:\overline{D}_{\delta^{1/d}}\to B\setminus\CC^\mu\]
by $\tau(z_1):=\beta(z_1^d)$. Since the coordinate functions of $\beta$ are algebraic power series at the origin, the
same holds for the coordinate functions of $\tau$.

Define a self-homeomorphism $\theta$ of ${D}_{\delta^{1/d}}\times\overline{B}_\epsilon\subset\CC^n$ by the formula
\[\theta(z_1,z_2,...,z_n):=(z_1,h_1(z_2,...,z_n,z_1^d)),\]
where $h_1$ is the homeomorphism predicted in Lemma~\ref{isotopy} for the parameter value $u=1$. 
Since $h_1$ preserves the stratification $\calA'$, and
the hypersurface $Z\subset\CC^{n-1}\times\CC^\mu$ is a union of strata of $\calA'$,
the hypersurface $X\subset\overline{D}_{\delta^{1/d}}\to B\setminus\CC^\mu$ is sent by $h_1$ to the hypersurface 
$X'\subset\overline{D}_{\delta^{1/d}}\to B\setminus\CC^\mu$ given by the pulback by $\tau$ of $Z$.

\subsection{}\label{final} In the previous steps we have reduced to the case in which the path
\[\gamma:(\CC,0)\to\CC^\mu\]
associated to $(f,z_1)$ in the base of the versal deformation of associated to $g$ is given by algebraic power series. This means
that if $(\gamma_1(z_1),...,\gamma_\mu(z_1))$ are the coordinate power series of $\gamma$ there exist polynomials in two variables
$P_1(y_1,z_1)$,...,$P_\mu(y_\mu,z_1)$ such that $P(\gamma_i(z_1),z_1)$ vanishes identically for any $1\leq i\leq\mu$. 

Consider coordinates $(y_1,...,y_\mu,z_1)$ in $\CC^{\mu}\times\CC$ and let $G$ be the curve defined by the ideal
$(P_1,...,P_\mu)$. Consider the natural projections
\[pr_1:\CC^{\mu}\times\CC\times\CC^{n-1}\to\CC^{\mu}\times\CC,\]
\[pr_2:\CC^{\mu}\times\CC\times\CC^{n-1}\to\CC^{\mu}\times\CC^{n-1},\]
\[pr_3:\CC^{\mu}\times\CC\times\CC^{n-1}\to\CC\times\CC^{n-1}.\]
Define the pure $(n-1)$-dimensional algebraic subset $Z'$ in $\CC^{\mu}\times\CC\times\CC^{n-1}$ by
\[Z':=pr_1^{-1}(G)\cap pr_2^{-1}(Z).\]  
Notice that, at the level of analytic germs, the restriction
\[pr_3|_{Z'}:(Z',O)\to (X,O)\]
is an isomorphism. If we could ensure that the restriction $pr_3|_{Z'}$ is proper and that the origin 
$O\in\CC^{\mu}\times\CC\times\CC^{n-1}$ is the only preimage of the origin 
$O\in\CC\times\CC^{n-1}$ by $pr_3|_{Z'}$ then we would be done, because then the germ $(X,0)$ would coincide with the algebraic 
hypersurface germ $(pr(Z'),O)$.

We will modify the projection $pr_3$ so we get the desired properties. Observe that the embedding dimension of
the germ $(Z',O)$ is $n$, since it is analytically isomorphic to the hypersurface germ $(X,O)$. Therefore its Zariski tangent
space $T$ is $n$-dimensional. Furthermore, observe that the projection $pr_3$ maps the Zariski tangent space $T$ isomorphically.
By dimension reasons, a generic projection
\[pr_4:\CC^{\mu}\times\CC\times\CC^{n-1}\to\CC^n\]
satisfies that it maps the Zariski tangent space $T$ isomorphically and that the restriction $pr_4|_{Z'}$ is proper and only 
has the origin in $\CC^{\mu}\times\CC\times\CC^{n-1}$ as preimage of the origin in $CC\times\CC^{n-1}$.

Let $(W,O)\subset\CC^{\mu}\times\CC\times\CC^{n-1}$ be a germ of smooth $n$-dimensional analytic space containing $(Z',O)$ and
such that its tangent space at the origin is equal to $T$. The fact that a projection $p$ maps
the Zariski tangent space $T$ isomorphically implies that it induces an analytic isomorphism of germs of pairs
\[p|_{W}:(W,Z',O)\to (\CC^n,p(Z'),O).\]
Applying this to the projections $pr_3|_{W}$ and $pr_4|_{W}$ we find that the germ $(\CC^n,X,O)$ is analytically isomorphic
to $(\CC^n,pr_4(Z'),O)$, which is an algebraic hypersurface germ. This concludes the proof.

\end{document}